\documentclass{amsart}
\usepackage{tikz}
\usepackage{xcolor}
\usepackage{amssymb,latexsym,amsmath,extarrows}
\usepackage{graphicx,mathrsfs,comment}
\usepackage{hyperref,url}

\usepackage{amstext}
\usepackage{bbm}

\numberwithin{equation}{section}

\DeclareMathOperator{\mes}{mes}

\newtheorem{theorem}{Theorem}[section]
\newtheorem{lemma}[theorem]{Lemma}

\newtheorem{proposition}[theorem]{Proposition}

\theoremstyle{remark}

\newcommand{\cI}{\mathcal I}

\newcommand{\dint}{\displaystyle\int}

\newcommand{\bh}{{\boldsymbol h}}
\newcommand{\bx}{{\bf x}}

\newcommand{\bz}{{\bf z}}

\newcommand{\fm}{{\mathfrak m}}
\newcommand{\fM}{{\mathfrak M}}

\newcommand{\cM}{{\mathcal M}}

\newcommand{\R}{\mathbb{R}}
\newcommand{\Z}{\mathbb{Z}}

\newcommand{\1}{\mathbbm{1}}

\newcommand{\fB}{\mathfrak{B}}
\newcommand{\balp}{\boldsymbol{\alpha}}
\newcommand{\bbet}{\boldsymbol{\beta}}

\begin{document}

\title[]{On the density of rational lines on diagonal cubic hypersurfaces, II}

\author{Scott T. Parsell} \address{ Scott T. Parsell\\  
Department of Mathematics, West Chester University, United States}\email{sparsell@wcupa.edu}

\author{Kiseok Yeon} \address{ Kiseok Yeon\\
Department of Mathematics, University of California, Davis, United States}\email{kyeon@ucdavis.edu}

\begin{abstract}

In this paper, we establish the expected asymptotic formula for the number of rational lines on a diagonal cubic hypersurface in 18 or more variables, improving on recent work of the second author. This is achieved via a refined mean value estimate for minor arcs that non-trivially exploits a shifting variables argument in both underlying dimensions. 


\end{abstract}

\maketitle


\section{Introduction}

The problem of establishing the existence of linear spaces on algebraic varieties has a long history, beginning with work of Brauer \cite{MR13127} and Birch \cite{MR97359}. The inductive nature of these arguments, however, does not immediately yield quantitative bounds on the number of variables required, and in fact the first general explicit bounds appeared many years later in work of Wooley \cite{MR1623369}. Subsequent work of Brandes \cite{MR3198749} establishes asymptotic formulae for the number of linear spaces on a variety defined by a system of homogeneous polynomial equations of the same degree, and Brandes \cite{MR4259461} further specializes to the case of rational lines on the hypersurface defined by a single equation.

The case in which the underlying variety is defined by a system of cubic forms has received significant attention, dating to work of Lewis and Schulze-Pillot \cite{MR0755988}. Refinements due to Wooley \cite{MR1458715} show for instance that the hypersurface defined by a cubic form in at least 37 variables contains a rational line. Brandes and Dietmann \cite{MR4268805} later showed that 31 variables suffice in the non-singular case, and the non-singularity hypothesis is removed in very recent work of Brandes, Dietmann, and Leep \cite{key2026}. 

When attempting to harness the full power of the circle method, one often restricts attention to the case of diagonal forms. Suppose that $k \ge 3$, and consider the hypersurface defined by the additive equation
\begin{equation}\label{hypersurface}
    \sum_{i=1}^sc_iz^k_i=0,
\end{equation}
where $c_i\ (1\leq i\leq s)$ are non-zero integers. Explicit results concerning linear spaces of arbitrary dimension satisfying (\ref{hypersurface}) have origins in work of Arkhipov, Karatsuba, and Chubarikov \cite{MR608411}, with later refinements occurring in \cite{MR2149529}, \cite{MR2485413}, and \cite{MR3132907}. The special case of lines on a diagonal hypersurface of degree $k$ is examined  in \cite{MR1817503} and \cite{MR4744752}, for example, and a detailed study of the cubic case began with work of the first author \cite{MR1778504}. 

To describe further progress, we set up some additional notation.
Our goal is to count pairs of vectors $\boldsymbol{x}=(x_1,\ldots,x_s)\in \mathbb{Z}^s$ and $\boldsymbol{y}=(y_1,\ldots,y_s)\in \mathbb{Z}^s$ such that the line $l: \boldsymbol{x}+t\boldsymbol{y}$ is contained in the hypersurface defined by $(\ref{hypersurface}).$
It follows from the Binomial Theorem that such pairs $(\boldsymbol{x}, \boldsymbol{y})$ 
are precisely the solutions of the system
\begin{equation}\label{system}
     \sum_{i=1}^sc_ix_i^{k-j} y_i^j =0 \qquad (0 \le j \le k).
\end{equation}
We therefore let $N_{s}(X):=N_{s,k}(X;\boldsymbol{c})$ denote the number of solutions $(\boldsymbol{x},\boldsymbol{y})\in \mathbb{Z}^{2s}\cap [-X,X]^{2s}$ to $(\ref{system})$. 

By making clever use of the cubic case of Vinogradov's mean value theorem due to Wooley \cite{MR3479572}, Zhao \cite{MR3552298} established the expected asymptotic formula for $N_{s,3}(X)$ whenever $s \ge 21$, improving on the bound $s \ge 29$ due to the first author \cite{MR2965969}. 
Following the work of Bourgain, Demeter, and Guo \cite{MR3709122} on the two-dimensional cubic Parsell-Vinogradov system, one can now recover Zhao's result in a routine manner by completing the system (\ref{system}) 
to obtain the required mean value estimate. Very recently, the second author \cite{ky2026} was able to save two additional variables by devising a multi-dimensional version of Wooley's shifting variables argument (see \cite{MR2913181}) that leverages the bounds of \cite{MR3709122} to obtain a mean value estimate restricted to minor arcs. The surplus minor arc savings permits an interpolation with lower moments of the relevant exponential sum, thus reducing the total number of variables required. 


The purpose of the present paper is to refine the shifting variables argument of \cite{ky2026} to more fully exploit the two-dimensional nature of the problem. 
Our main achievement is a stronger mean value estimate that exploits non-trival summation over both integer shifts, subject to a suitable minor arc condition. It transpires that this allows us to save an extra variable, and we thus obtain the following result.

\begin{theorem}\label{cubic theorem}
    Whenever $s\geq 18$, one has
    \begin{equation*}\label{cubic AF}
        N_{s,3}(X)=\sigma X^{2s-12}+O(X^{2s-12-\tau}),
    \end{equation*}
    for some $\tau>0$, where $\sigma =\sigma_{\boldsymbol{c}}$ is a positive constant depending on the coeffcients $\boldsymbol{c}$.
\end{theorem}

As indicated above, this sharpens \cite[Theorem 1.1]{ky2026}, where the same conclusion was obtained for $s \ge 19$. Here $\sigma$ represents a product of local densities, which is positive for $s \ge 14$ as a consequence of \cite[Lemma 5.1]{MR1778504}.





We mention that a straightforward generalization of our methods, combined with the results of Guo and Zhang \cite{MR3994585} on two-dimensional Parsell-Vinogradov systems, produces analogous results for larger $k$. In this instance, we would save roughly $k/3$ variables over what follows directly from \cite{MR3994585}, in which the total number of variables required to obtain the expected asymptotic formula for $N_{s,k}(X)$ is of order $k^3$. Since the savings is not especially compelling and does not illuminate any new ideas, we choose not to provide details in this direction.

Throughout we use Vinogradov's notation, where $f\ll g$ means that $f(x)\leq Cg(x)$ for some sufficiently large constant $C>0.$ We may use $f=O(g)$ with the same meaning. We also make liberal use of vector notation, writing $a \le |\bx| \le b$ to mean that each component of $\bx$ lies in $[a,b]$. Finally, we use $\varepsilon$ to denote a small positive number, whose value may change from statement to statement.

 \section{Mean value estimates on minor arcs}\label{sec2}
 
 Our main object of interest is the exponential sum
\begin{equation}\label{deg k sum}    F(\boldsymbol{\alpha}):=\sum_{1\leq x,y\leq X}e(\alpha_1x^3+\alpha_2x^{2}y+\alpha_3xy^{2}+\alpha_{4}y^3),
\end{equation}
with $\boldsymbol{\alpha}=(\alpha_1,\alpha_2,\alpha_3,\alpha_{4}).$ 
For $s\in \mathbb{N}$ we further define
$J_s(X)=J_{s,2,3}(X)$ to be the number of integer solutions of the system
\begin{equation*}
\sum_{i=1}^s x_i^{d-j}y_i^{j} = \sum_{i=s+1}^{2s} x_i^{d-j} y_i^{j} \qquad (0 \le j \le d, \  1 \le d \le 3)
\end{equation*}
satisfying $1\le \boldsymbol{x},\boldsymbol{y} \le X$. 
In this section we aim to estimate restricted mean values of the exponential sum $(\ref{deg k sum})$ by relating them to $J_s(X)$, for which sharp bounds are available  
due to Bourgain, Demeter, and Guo \cite{MR3709122}. 
The particular estimate we need here, which follows from \cite[Theorem 1.5]{MR3709122}, is
\begin{equation}\label{J10}
J_{10}(X) \ll X^{20+\varepsilon}.
\end{equation}
In the shifting variables argument of \cite{ky2026} one replaces $x$ by $x-z_1$ and $y$ by $y-z_2$ and averages over $z_1$ and $z_2$. In analyzing the effects of these shifts on the cubic equations, it transpires that the linear forms
\begin{equation}\label{def L_1L_2}
\begin{aligned}
 L_1(\boldsymbol{h},\boldsymbol{\alpha})&:= 
 3h_1\alpha_1+2h_2\alpha_2+h_3\alpha_3 \\
  L_2(\boldsymbol{h},\boldsymbol{\alpha})&:=
  h_1\alpha_2+2h_2\alpha_3+3h_3\alpha_4
\end{aligned}
\end{equation}
play a prominent role, as they arise naturally from the partial translation invariance inherent in a Vinogradov-type system with only the quadratic equations inhomogeneous. It is therefore useful to define the corresponding exponential sum 
\begin{equation}\label{T def}
\begin{aligned}
T(\boldsymbol{\alpha},\boldsymbol{\beta}):=\sum_{1\leq z_1,z_2\leq X}\biggl|\sum_{|\bh|\leq sX^{2}}e(-z_1L_1(\bh,\boldsymbol{\alpha})-z_2L_2(\bh,\boldsymbol{\alpha})-\boldsymbol{\beta}\cdot \bh)\biggr|,
\end{aligned}
\end{equation} 
where $\balp \in [0,1)^4$ and $\bbet \in [0,1)^3$
From the argument of \cite[Lemma 3.1]{ky2026}, we extract the following fundamental lemma, which relates mean values of $F(\boldsymbol{\alpha})$ to pointwise estimates for $T(\boldsymbol{\alpha},\boldsymbol{\beta})$.

\begin{lemma}\label{reduction to T}
Let $s$ be a natural number, and let $\fB$ denote any measurable subset of $[0,1)^{4}$. Then one has
\begin{equation*}
\int_{\fB}|F(\boldsymbol{\alpha})|^{2s}\,d\boldsymbol{\alpha} \ll (\log X)^{4s} \, J_s(2X) \sup_{\substack{\boldsymbol{\alpha}\in \fB\\ \boldsymbol{\beta}\in [0,1)^3 }} |T(\boldsymbol{\alpha},\boldsymbol{\beta})|.
\end{equation*}
\end{lemma}

\begin{proof}
On inspecting the proof of \cite[Lemma 3.1]{ky2026}, one finds that the argument leading to the inequality (3.21) is unchanged if the set ${\mathcal M}_l(H)$ is replaced by an arbitrary measurable set ${\mathfrak B} \subseteq [0,1)^4$. Hence the statement of the lemma follows immediately upon inserting the bound (3.22) from \cite{ky2026} into (3.21).
\end{proof}


We now seek mean value estimates on various sets of minor arcs, and Lemma \ref{reduction to T} reduces the problem to obtaining suitable estimates for $T(\boldsymbol{\alpha}, \boldsymbol{\beta})$. 
We start by defining a set of one-dimensional major arcs
\begin{equation*}
\fM(Q) = \bigcup_{\substack{1 \le a \le q \le Q \\ (q,a)=1}} \fM_{q,a}(Q),    
\end{equation*}
where
\begin{equation*}
\fM_{q,a}(Q) = \{\alpha \in [0,1): |q\alpha - a| \le Q X^{-3}\}.
\end{equation*}
Further define $\fm(Q) = [0,1) \setminus \fM(Q)$, and for $l = 1,2,3,4$ write
\begin{equation*}
\cM_l(Q) = \{\boldsymbol{\alpha} \in [0,1)^{4} : \alpha_l \in \fm(Q)\}.
\end{equation*}
By inserting the bound for $T$ obtained in (3.27) of \cite{ky2026}, that is for $1\leq l\leq 4$ one has
\begin{equation}\label{basic exponential sum}
     \sup_{\substack{\boldsymbol{\alpha}\in \mathcal{M}_l(Q)\\ \boldsymbol{\beta}\in [0,1)^3 }}|T(\boldsymbol{\alpha},\boldsymbol{\beta})|\ll X^{8+\varepsilon}(Q^{-1}+X^{-1}),
\end{equation}
into Lemma \ref{reduction to T}, we recover the estimate of \cite[Lemma 3.1]{ky2026}.

\begin{lemma}\label{basic minor}
    Let $Q$ and $X$ be positive numbers with $Q \le X^{3/2}$, and suppose that $s$ is a natural number. Then for $1 \le l \le 4$ one has
    \begin{equation*}\label{meanvalue lemma 2.2}
        \int_{\cM_l(Q)} |F(\boldsymbol{\alpha})|^{2s}\, d\boldsymbol{\alpha} \ll X^{8} \cdot J_s(2X) \cdot (Q^{-1} + X^{-1})(\log X)^{4s+1}. 
     \end{equation*}
\end{lemma}

\begin{proof}
This is essentially the statement of \cite[Lemma 3.1]{ky2026}. We note here that the factor of $X^{\varepsilon}$ occurring at (3.24) may in fact be replaced by $\log(2r) \ll \log X$, and the result follows on inserting the ensuing bound (3.27) into our Lemma \ref{reduction to T}.
\end{proof}

Our objective is now to establish a sharper bound for $T(\boldsymbol{\alpha},\boldsymbol{\beta})$ than that provided in Lemma \ref{basic minor} by employing a further dissection of the minor arcs 
 $\mathcal{M}_l(Q)$ into secondary major and minor arcs.
 Before proceeding, we offer a brief heuristic motivation for this refinement. 
 Recall the definitions (\ref{def L_1L_2}) and (\ref{T def}).
Suppose, hypothetically, that the term $2h_2\alpha_2$ were absent from $L_1(\boldsymbol{h},\boldsymbol{\alpha}),$ so that
\begin{equation*}
L_1(\boldsymbol{h},\boldsymbol{\alpha}) =3h_1 \alpha_1 + h_3\alpha_3.
\end{equation*}
Then, it follows by summing over $h_1$ and $h_2$ that the exponential sum $T(\boldsymbol{\alpha},\boldsymbol{\beta})$ would satisfy
\begin{equation*}
    T \ll X^2\sum_{1\leq z_1,z_2\leq X}\min\left\{X^2,\frac{1}{\|3\alpha_1z_1+\alpha_2z_2+\beta_1\|}\right\}\min\left\{X^2,\frac{1}{\|2\alpha_3z_2+\beta_2\|}\right\}.
\end{equation*}
Furthermore, if $\alpha_1$ and $\alpha_3$ belong to the minor arcs $\mathfrak{m}(Q)$, then by summing over $z_1$ and $z_2$ in succession and following the standard argument \cite[Lemma 3.2]{MR865981}, we would obtain
\begin{equation*}
T(\boldsymbol{\alpha},\boldsymbol{\beta})\ll X^{8}(Q^{-1}+X^{-1})^2(\log X)^2.
\end{equation*}
This represents a superior bound to (\ref{basic exponential sum}), provided that both $\alpha_1$ and $\alpha_3$ are in $\mathfrak{m}(Q).$
In the actual case where $2h_2\alpha_2$ is present, we can approximate this behavior when $\alpha_2$ lies in a suitable major arc. Let $q_2$ be a natural number such that $\|q_2\alpha_2\|$ is sufficiently small, and write $h_2=q_2l+r$ with $0\leq r\leq q_2-1$. When summing over $h_1$ and $l$, the term $2h_2\alpha_2=2(q_2l+r)\alpha_2$ remains nearly constant $(\text{mod}\ 1)$ because $\|q_2\alpha_2\|$ is small. This stability allows us to apply the aforementioned argument over the variables $h_1$
  and $l$ to achieve additional savings, provided $\boldsymbol{\alpha}$ lies on suitably defined secondary minor arcs.

In Lemma \ref{improved minor} below, we formalize this heuristic to provide a refined bound for $T(\boldsymbol{\alpha},\boldsymbol{\beta})$. Before stating the result, we define the relevant secondary major and minor arc dissections and introduce the necessary notation. 
Consider a fixed $H \le X^{3/2}$ and $L \le H^{1/2}$.
We define a modified set of major arcs by
\begin{equation*}
\mathfrak{M}(L;H):=\bigcup_{\substack{0\leq a\leq r\leq L\\ (r,a)=1}}\mathfrak{M}_{r,a}(L;H),
\end{equation*}
    where 
    \begin{equation*}
        \mathfrak{M}_{r,a}(L;H):=\{\beta\in [0,1):\ |r\beta-a|\leq LHX^{-3}\}.
    \end{equation*}
For each $\alpha \in \mathfrak{M}(H)$, let $q=q_{\alpha}$ denote the smallest natural number for which $\|q\alpha\|\leq HX^{-3}$, and define 
\begin{equation*}
    \mathfrak{L}^{(\alpha)}(L;H):=\left\{\beta\in [0,1):\ q\beta \ (\text{mod}\ 1) \in \mathfrak{M}(L;H)\right\}.
\end{equation*}
Additionally, define 
\begin{equation}\label{secondary minor arcs}
 \mathfrak{l}^{(\alpha)}(L; H)=[0,1)\setminus \mathfrak{L}^{(\alpha)}(L; H).   
\end{equation} It is worth emphasizing that these sets $\mathfrak{L}^{(\alpha)}(L; H)$ and $\mathfrak{l}^{(\alpha)}(L; H)$ depend on $\alpha\in \mathfrak{M}(H).$ 
With above definitions in mind, for $m,n\in \mathbb{N}$ with $1\leq m,n\leq 4,$ we define
\begin{equation}\label{def of Lmn(L)}
    \mathcal{L}^*_{m,n}(L; H):=\{\balp \in [0,1)^{4}:\ \alpha_m \in {\mathfrak M}(H), \ \alpha_n\in \mathfrak{l}^{(\alpha_m)}(L; H) \}.
\end{equation}
Define 
\begin{equation}\label{def of Ll(L;H)}
    \begin{aligned}
        \mathcal{L}_l(L;H) &=\left\{\begin{aligned}
            &\mathcal{L}_{2,3}^*(L; H)  \ \ \ \text{when}\ l=1, 3\\
            &\mathcal{L}_{3,2}^*(L; H)\ \ \ \text{when}\ l=2, 4.\\
        \end{aligned}\right.
    \end{aligned}
\end{equation}

Let $Q, H, L$, and $X$ be positive numbers with $Q \leq H\leq X^{3/2}$ and $L\leq H^{1/2}$. We shall show in the proof of Lemma \ref{improved minor} that the exponential sum $T(\boldsymbol{\alpha},\boldsymbol{\beta})$ has a sharper bound over $\boldsymbol{\alpha}\in\mathcal{M}_l(Q)\cap \mathcal{L}_{l}(L;H)$ with $1\leq l\leq 4,$ compared to the bound (\ref{basic exponential sum}). Eventually, in Lemma \ref{improved minor}, we shall have a sharper mean value estimate for $F(\boldsymbol{\alpha})$ over the set $\mathcal{M}_l(Q)\cap \mathcal{L}_{l}(L;H)$, compared to 
Lemma $\ref{basic minor}$.

Prior to this, we verify that the measure of the set on which we're forced to use the weaker estimate (\ref{basic exponential sum}) for $T$ embedded in Lemma \ref{basic minor} is indeed small. 

\begin{proposition}\label{pro5.2}
Let $H,L$ and $X$ be positive numbers with $H\leq X^{3/2}$ and $L\leq H^{1/2}$. 
  Consider the set
  \begin{equation*}
  \mathfrak{L}(L;H) :=  \{(\alpha,\beta)\in [0,1)^2:\ \alpha\in \mathfrak{M}(H),\ \beta\in \mathfrak{L}^{(\alpha)}(L;H)\}.
  \end{equation*}
  Then, the set  $\mathfrak{L}(L;H)$ is measurable and 
  \begin{equation*}
      \mes{\left(  \mathfrak{L}(L;H)\right)}\ll   L^2H^3X^{-6}.
  \end{equation*}
\end{proposition}
\begin{proof}
  First, we claim that for a given $\alpha\in \mathfrak{M}(H)$, the measure of $\mathfrak{L}^{(\alpha)}(L;H)$ is bounded above by $L^2HX^{-3}.$ 
Recall the denominator $q=q_{\alpha}.$ With this, let $A \subseteq[0,q)$ be a set such that $x \in A$ if and only if
 $x\ (\text{mod}\ 1)$ is in $\mathfrak{M}(L;H).$ Observe that this set $A$ has the property that 
    $$\text{mes}(A)=q\cdot\text{mes}(\mathfrak{M}(L;H)).$$ Meanwhile, we see by the definition of $\mathfrak{L}^{(\alpha)}(L;H)$ that $\mathfrak{L}^{(\alpha)}(L;H)=\frac{1}{q}A$. Therefore, we conclude that
    \begin{equation}\label{bound for the measure of LalphaL}
    \begin{aligned}
        \text{mes}\bigl(\mathfrak{L}^{(\alpha)}(L;H)\bigr)&=\frac{1}{q}\text{mes}(A)\\
        &=\frac{1}{q}\cdot q\cdot\text{mes}(\mathfrak{M}(L;H))\\
        &=\text{mes}(\mathfrak{M}(L;H))\\
        &\leq L^2HX^{-3}.
    \end{aligned}
    \end{equation}

    Next, we shall show that  the set $ \mathfrak{L}(L;H)$ is also measurable set. Consider the decomposition
    \begin{equation*}
        \mathfrak{M}(H)=\bigcup_{\substack{1\leq n\leq H}}A_{n},
    \end{equation*}
    where 
    \begin{equation*}
        A_{n}=\{\alpha\in \mathfrak{M}(H):\ q_{\alpha}=n\}.
    \end{equation*}
 Notice that each set $A_{n}$ is a measurable set and $A_{n}$ are pairwise disjoint for all $n$.
  Therefore, we deduce that
 \begin{equation}\label{countable decomposition}
 \begin{aligned}
     &\text{mes}\left(  \mathfrak{L}(L;H)\right)\\
     &=\text{mes}\left( \bigcup_{\substack{1\leq n\leq H}} \{(\alpha,\beta)\in [0,1)^2:\ \alpha\in A_n,\ \beta\in \mathfrak{L}^{(\alpha)}(L;H)\}\right)\\
     &=\sum_{\substack{1\leq n\leq H}}\text{mes}\left(\{(\alpha,\beta)\in [0,1)^2:\ \alpha\in  A_n,\ \beta\in \mathfrak{L}^{(\alpha)}(L;H)\}\right).
 \end{aligned}
 \end{equation}
 Meanwhile, we see that the dependence of the set $\mathfrak{L}^{(\alpha)}(L;H)$ on $\alpha$ is related only to natural numbers $q_{\alpha}\in \mathbb{N}$. Hence, we find that for fixed $n\in [1,H]\cap \mathbb{N}$, the sets $\mathfrak{L}^{(\alpha)}(L;H)$ are equal for all $\alpha\in A_{n}$. Then, we infer by the bound (\ref{bound for the measure of LalphaL}) that
 \begin{equation}\label{measure bound 2}
 \begin{aligned}
 &\text{mes}\left(\{(\alpha,\beta)\in [0,1)^2:\ \alpha\in A_n,\ \beta\in \mathfrak{L}^{(\alpha)}(L;H)\}\right)\\
     &\leq L^2HX^{-3}\cdot \text{mes}(A_{n}).
 \end{aligned}
 \end{equation}
 Thus, we conclude by (\ref{countable decomposition}) and $(\ref{measure bound 2})$ that 
 \begin{equation*}
 \begin{aligned}
   \text{mes}\left(  \mathfrak{L}(L;H)\right)&=\sum_{\substack{1\leq n\leq H}}\text{mes}\left(\{(\alpha,\beta)\in [0,1)^2:\ \alpha\in  A_n,\ \beta\in \mathfrak{L}^{(\alpha)}(L;H)\}\right)\\
     &\leq  \sum_{\substack{1\leq n\leq H}}L^2HX^{-3}\cdot \text{mes}(A_{n})\\
     &\leq L^2HX^{-3}\cdot \text{mes}(\mathfrak{M}(H))\\
     &\leq L^2H^3X^{-6},
 \end{aligned}
 \end{equation*}
 which completes the proof.
\end{proof}

Let $Q$ be a positive number. Since $\mathcal{M}_l(Q)$ is measurable, we infer by Proposition \ref{pro5.2} together with the definition of $\mathcal{L}_{l}(L;H)$ that $\mathcal{M}_l(Q)\cap \mathcal{L}_{l}(L;H)$ is measurable. The following lemma provides an upper bound for the mean value of $F(\boldsymbol{\alpha})$ over the set $\mathcal{M}_l(Q)\cap \mathcal{L}_{l}(L;H)$, which is a sharper bound than that in Lemma \ref{basic minor}.

\begin{lemma}\label{improved minor}
    Let $Q, H, L$, and $X$ be positive numbers with $Q \leq H\leq X^{3/2}$ and $L\leq H^{1/2}$, and suppose that $s $ is a natural number. Then for $1\leq l\leq 4$ one has 
    \begin{equation*}\label{sharp bound for T}
    \begin{aligned}
        & \int_{\mathcal{M}_l(Q)\cap \mathcal{L}_{l}(L;H)} |F(\boldsymbol{\alpha})|^{2s}\, d\boldsymbol{\alpha}  \\
        & \ \ \ \ \ \ \ \ll X^{8} \cdot J_s(2X)\cdot \left(Q^{-1}+X^{-1}\right)\bigl(L^{-1}+X^{-1/2}\bigr)(\log X)^{4s+2}.
    \end{aligned}
    \end{equation*}
\end{lemma}

\begin{proof}
We see by reversing the roles of $x$ and $y$ in the exponential sum $F(\boldsymbol{\alpha})$ that the bound for the cases $l=3,4$ immediately follows by the bound for the cases $l=1,2$.  Thus, it suffices to restrict our attention to the latter cases.

We shall consider the cases $l=1$ and $l=2$ separately. 
First, consider the mean value
\begin{equation*}\label{mean value l=1}
   \cI_1 := \int_{\mathcal{M}_1(H)\cap \mathcal{L}^*_{2,3}(L;H)}|F(\boldsymbol{\alpha})|^{2s}\, d\boldsymbol{\alpha}.
\end{equation*}
It follows by Lemma $\ref{reduction to T}$ with $\mathfrak{B}=\mathcal{M}_1(Q)\cap \mathcal{L}^*_{2,3}(L;H)$ that 
\begin{equation}\label{bound for mean value over M_1cap L(2,3)}
\cI_1 \ll  \, (\log X)^{4s} \, J_s(2X) \sup_{\substack{\boldsymbol{\alpha}\in \mathcal{M}_1(Q)\cap \mathcal{L}^*_{2,3}(L;H)\\ \boldsymbol{\beta}\in [0,1)^3 }} |T(\boldsymbol{\alpha},\boldsymbol{\beta})|.
\end{equation}
 Recall the definition (\ref{def of Lmn(L)}) of $\mathcal{L}^*_{2,3}(L;H)$. For a given $\alpha_2\in \mathfrak{M}(H),$ recall the definition of $q=q_{\alpha_2}$, so that $\|q\alpha_2\| \le HX^{-3}$.  With this $q$,  write  
 \begin{equation}\label{representation h2}
     h_2= q\lfloor X^{2}/H\rfloor r_1+qr_2+r_3,
 \end{equation} for some $r_1,r_2,r_3\in \Z$ with 
\begin{equation}\label{ranges of r1r2r3}
    \begin{aligned}
        -sH/q\leq &r_1\leq sH/q\\
        0\leq &r_2\leq \lfloor X^{2}/H\rfloor-1\\
        0\leq &r_3\leq q-1.
    \end{aligned}
\end{equation}
Notice here that $r_1,r_2,r_3\in \Z$ are uniquely determined for a given $h_2\in\Z$ with $|h_2|\leq sX^{2}$. 
Recall the definition (\ref{T def}) of $T(\boldsymbol{\alpha},\boldsymbol{\beta})$, and write the summation over $h_2$ in $T(\boldsymbol{\alpha},\boldsymbol{\beta})$ as three sums over $r_1,r_2$ and $r_3$ using the representation $(\ref{representation h2}).$ Then, for a given $\alpha_2\in \mathfrak{M}(H)$,  one has
\begin{equation}\label{main estimate}
    T(\balp, \bbet) \ll X^2 \sum_{1\leq z_1,z_2\leq X} \bigl| S_1(\balp, \bbet, \bz) \, S_2(\balp, \bbet, \bz) \bigr|, 
\end{equation}
where
\begin{equation}\label{S1bd}
\begin{aligned}
S_1(\balp, \bbet, \bz) &:= \sum_{|h_1|\leq sX^{2}}e(-3\alpha_1h_1z_1-\alpha_2h_1z_2-\beta_1h_1) \\
&\ll \min\left\{X^{2},\frac{1}{\|3z_1\alpha_1+z_2\alpha_2+\beta_1\|}\right\}
\end{aligned}
\end{equation}
and
\begin{equation}\label{S2bd}
\begin{aligned}
S_2(\balp, \bbet, \bz) &:= \sum_{|h_2|\leq sX^2}e(-2\alpha_2h_2z_1-2\alpha_3h_2z_2-\beta_2h_2)  \\
&\ll \sum_{r_1,r_3}\biggl|\sum_{r_2}e(-2q\alpha_2r_2z_1-2q\alpha_3r_2z_2-q\beta_2r_2)\biggr|. \end{aligned}
\end{equation}
We remark here that the range of the summation over $r_2$ in the second inequality of (\ref{S2bd}) may not be over the full interval $0\leq r_2\leq \lfloor X^{2}/H\rfloor-1$ as described in $(\ref{ranges of r1r2r3}).$ However, for given $r_1$ and $r_3$, one infers by the construction of $r_1,r_2,r_3$ that the range of $r_2$ is a subinterval of $[0,\lfloor X^{2}/H\rfloor-1]$. Nevertheless, we did not specify the range because this choice of subintervals does not harm the argument.

Furthermore, we infer that 
\begin{equation}\label{S2bd2}
    \begin{aligned}
S_2(\balp, \bbet, \bz) &\ll \sum_{r_1,r_3}\biggl|\sup_{I\subseteq \left[1,X^{2}/H\right]}\sum_{r_2\in I}e(-2q\alpha_3r_2z_2-q\beta_2r_2)\biggr| \\
&\ll H \cdot \min\left\{\frac{X^{2}}{H},\frac{1}{\|2q\alpha_3z_2+q\beta_2\|}\right\}.
\end{aligned}
\end{equation}
Here the supremum of $I$ is over subintervals in $\left[1,X^2/H\right]$, and we used summation by parts in order to eliminate the factor $-2q\alpha_2r_2z_1$, together with the fact that
\begin{equation*}
    \|2q\alpha_2z_1\|\leq 2X\|q\alpha_2\|\leq 2HX^{-2}.
\end{equation*}


Since $\alpha_1\in \mathfrak{m}(Q)$, by the same argument leading from (3.24) to (3.27) in \cite{ky2026}, together with the proof of Lemma \ref{basic minor}, one deduces that 
\begin{equation}\label{almost last}
     \sum_{1\leq z_1\leq X}\min\left\{X^{2},\frac{1}{\|3z_1\alpha_1+z_2\alpha_2+\beta_1\|}\right\}\ll X^{3}(Q^{-1}+X^{-1})\log X.
\end{equation}
Thus by summing over $z_1$ in $(\ref{main estimate}),$ it follows from (\ref{S1bd}), (\ref{S2bd}),(\ref{S2bd2}), and $(\ref{almost last})$ that for a given $\boldsymbol{\alpha}\in \mathcal{M}_1(Q)\cap \mathcal{L}^*_{2,3}(L;H)$ and $\boldsymbol{\beta}\in [0,1)^3$, one has
\begin{equation}\label{almost last bound for T}
    \begin{aligned}
&T(\boldsymbol{\alpha},\boldsymbol{\beta})\\
&\ll HX^{5}\left(Q^{-1}+X^{-1}\right)\log X\sum_{1\leq z_2\leq X}\min\left\{\frac{X^{2}}{H},\frac{1}{\|2q\alpha_3z_2+q\beta_2\|}\right\}.
    \end{aligned}
\end{equation}
Now by Dirichlet's approximation theorem, there exists $r\in \mathbb{N}$ such that $\|rq\alpha_3\|\leq LHX^{-3}$ and $r\leq X^3H^{-1}L^{-1}$. Hence, the fact, that $\alpha_3\in \mathfrak{l}^{(\alpha_2)}(L;H)$, ensures that $r>L.$
Hence by invoking \cite[Lemma 3.2]{MR865981} as in the corresponding argument leading to $(\ref{almost last})$, one has
\begin{equation}\label{extrasavin}
    \sum_{1\leq z_2\leq X}\min\left\{\frac{X^{2}}{H},\frac{1}{\|2q\alpha_3z_2+q\beta_2\|}\right\}\ll \frac{X^{3}}{H}\bigl(L^{-1}+X^{-1}+H{X^{-2}}\bigr)\log X,
\end{equation}
for a given $\alpha_3\in \mathfrak{l}^{(\alpha_2)}(L;H)$. By substituting $(\ref{extrasavin})$ into the upper bound in $(\ref{almost last bound for T})$ and recalling that $H \le X^{3/2}$, 
we see that 
\begin{equation}\label{bound1 for T(alpha,beta)}
\begin{aligned}
    \sup_{\substack{\boldsymbol{\alpha}\in \mathcal{M}_1(Q)\cap \mathcal{L}^*_{2,3}(L;H)\\ \boldsymbol{\beta}\in [0,1)^3 }} |T(\boldsymbol{\alpha},\boldsymbol{\beta})|\ll X^{8}\left(Q^{-1}+X^{-1}\right)\bigl(L^{-1}+{X^{-1/2}}\bigr)(\log X)^2.
\end{aligned}
\end{equation}
Therefore, on substituting $(\ref{bound1 for T(alpha,beta)})$ into $(\ref{bound for mean value over M_1cap L(2,3)})$, the conclusion of the lemma follows in the case $l=1$. 

Next, for the case $l=2$, we consider the mean value 
\begin{equation*}\label{mean value l>1}
    \cI_2 := \int_{\mathcal{M}_2(Q)\cap \mathcal{L}^*_{3,2}(L;H)}|F(\boldsymbol{\alpha})|^{2s}\, d\boldsymbol{\alpha}.
\end{equation*}
Since the following argument is almost the same as the case $l=1,$ we proceed quickly. It follows again by Lemma $\ref{reduction to T}$ with $\mathfrak{B}=\mathcal{M}_2(Q)\cap \mathcal{L}^*_{3,2}(L;H)$ that 
\begin{equation}\label{bound for mean value over Mlcap L(l+1,l)}
   \cI_2 \ll  \, (\log X)^{4s} \, J_s(2X) \sup_{\substack{\boldsymbol{\alpha}\in \mathcal{M}_2(Q)\cap \mathcal{L}^*_{3,2}(L;H)\\ \boldsymbol{\beta}\in [0,1)^3 }} |T(\boldsymbol{\alpha},\boldsymbol{\beta})|.
\end{equation}
Recall the definition of $\mathcal{L}^*_{3,2}(L;H)$, and for a given $\alpha_{3}\in \mathfrak{M}(H),$ recall the definition of $q:=q_{\alpha_{3}}$.  With this $q$,  we again represent $h_2$ in the form $(\ref{representation h2})$ with
the constraints (\ref{ranges of r1r2r3}).
As in the previous argument for the case $l=1$, we replace the summation over $h_{2}$ in $T(\boldsymbol{\alpha},\boldsymbol{\beta})$ by three sums over $r_1,r_2$ and $r_3$. 
Then, for a given $\alpha_{3}\in \mathfrak{M}(H)$, it follows by the same argument in (\ref{main estimate})--(\ref{S2bd2}) that
\begin{equation*}
\begin{aligned}
&T(\boldsymbol{\alpha},\boldsymbol{\beta})\\
&\ll HX^{2} \! \! \sum_{1\leq z_1,z_2\leq X} \! \! \min\left\{X^{2},\frac{1}{\|3z_1\alpha_{1}+z_2\alpha_2+\beta_{1}\|}\right\} \min\left\{\frac{X^{2}}{H},\frac{1}{\|2q\alpha_2z_1+q\beta_2\|}\right\}.
\end{aligned}
\end{equation*}

 By summing first over $z_2$ in the argument leading from $(\ref{almost last})$ to $(\ref{almost last bound for T})$, we deduce that for a given $\boldsymbol{\alpha}\in \mathcal{M}_2(Q)\cap \mathcal{L}^*_{3,2}(L;H)$ and $\boldsymbol{\beta}\in [0,1)^3$ one has
 \begin{equation*}
 \begin{aligned}
T(\boldsymbol{\alpha},\boldsymbol{\beta})\ll HX^{5}\left(Q^{-1}+X^{-1}\right)\log X\sum_{1\leq z_1\leq X}\min\left\{\frac{X^{2}}{H},\frac{1}{\|2q\alpha_2z_1+q\beta_2\|}\right\}.
 \end{aligned}
 \end{equation*}
Then by the same argument leading from $(\ref{almost last bound for T})$ to $(\ref{bound1 for T(alpha,beta)})$, we infer that
\begin{equation}\label{bound 2 for T(alpha,beta)}
\begin{aligned}
\sup_{\substack{\boldsymbol{\alpha}\in \mathcal{M}_2(Q)\cap \mathcal{L}^*_{3,2}(L;H)\\ \boldsymbol{\beta}\in [0,1)^3 }} |T(\boldsymbol{\alpha},\boldsymbol{\beta})|\ll X^{8}\left(Q^{-1}+X^{-1}\right)\bigl(L^{-1}+{X^{-1/2}}\bigr)(\log X)^2.
\end{aligned}
\end{equation}

By substituting (\ref{bound 2 for T(alpha,beta)}) into $(\ref{bound for mean value over Mlcap L(l+1,l)})$, the conclusion of the lemma follows in the case $l=2$, and this suffices to complete the proof.
\end{proof}

In the next section we perform an interpolation between the mean values in Lemmas \ref{basic minor} and \ref{improved minor} with $s=10$ and lower moments containing measure-theoretic information. The key point is that we are able to use the stronger estimate contained in Lemma \ref{improved minor} unless $(\alpha_2, \alpha_3)$ lies in a set of unusually small measure controlled via Proposition \ref{pro5.2}. As in \cite{ky2026}, the 8th moment suffices for our purposes.


\begin{lemma}\label{k-1 dim measure}
    Suppose that $\mathfrak{D}$ is a measurable set in $[0,1)^{2}$.  Then one has
    \begin{equation*}
        \int_0^1\int_{\mathfrak{D}}\int_0^1 |F(\boldsymbol{\alpha})|^{8} \, d\boldsymbol{\alpha}\ll X^{10+\varepsilon}\cdot \mes{(\mathfrak{D})},
    \end{equation*}
where $d\boldsymbol{\alpha}=d\alpha_1 \, d\alpha_2 \, d\alpha_3 \, d\alpha_4.$
\end{lemma}

\begin{proof}
This is \cite[Lemma 4.1]{ky2026}.
\end{proof}

\bigskip

\section{Pruning and Endgame}\label{sec3}

We now have the tools needed to prove Theorem \ref{cubic theorem}. Define
\begin{equation*}\label{F_{cj}}
    F_{c_i}(\boldsymbol{\alpha}) =\sum_{|x|,|y|\leq X}e(c_i\alpha_1x^3+c_i\alpha_2x^{2}y+ c_i\alpha_3xy^2+c_i\alpha_{4}y^3),
\end{equation*}
where $\boldsymbol{\alpha}=(\alpha_1,\alpha_2, \alpha_3,\alpha_{4})\in \mathbb{R}^{4}$. It then follows by orthogonality that
\begin{equation*}
    N_s(X)=\int_{[0,1)^{4}} \prod_{i=1}^s F_{c_i}(\boldsymbol{\alpha}) \, d\boldsymbol{\alpha}.
\end{equation*}
To set up an application of the Hardy-Littlewood method, we let $\delta = 10^{-10}$ and define the major arcs $\mathfrak{N}_{\delta}$ by
\begin{equation*}
    \mathfrak{N}_{\delta}=\bigcup_{q\leq X^{\delta}}\bigcup_{\substack{1\leq \boldsymbol{a}\leq q\\ (q,\boldsymbol{a})=1}}\mathfrak{N}_{q,\boldsymbol{a}},
\end{equation*}
where
\begin{equation*}
    \mathfrak{N}_{q,\boldsymbol{a}} =\{\boldsymbol{\alpha}\in [0,1)^{4}:\ |\boldsymbol{\alpha}-\boldsymbol{a}/q|\leq X^{\delta-3}\}.
\end{equation*}
Upon writing $\mathfrak{n}_{\delta}=[0,1)^{4}\setminus \mathfrak{N}_{\delta}$ for the minor arcs, we have
\begin{equation}\label{major and minor arcs dissection}
    N_s(X)=\int_{\mathfrak{N}_{\delta}} \prod_{i=1}^s F_{c_i}(\boldsymbol{\alpha})\, d\boldsymbol{\alpha}+\int_{\mathfrak{n}_{\delta}} \prod_{i=1}^s F_{c_i}(\boldsymbol{\alpha}) \, d\boldsymbol{\alpha}.
\end{equation}
We can determine the major arc contribution by referencing existing arguments in the literature as in \cite{ky2026}. To describe this conclusion, we first define
\begin{equation*}
S(q,\boldsymbol{a})=\sum_{x=1}^q\sum_{y=1}^qe\left(\frac{a_1x^3+a_2x^{2}y+a_3xy^{2}+a_{4}y^3}{q}\right)
\end{equation*}
and 
\begin{equation*}
    S(q)=\sum_{\substack{1\leq \boldsymbol{a}\leq q\\(q,\boldsymbol{a})=1}}\prod_{i=1}^s(q^{-2}S(q,c_i\boldsymbol{a})).
\end{equation*}
We then define the singular series $\mathfrak{S}$ by
\begin{equation*}
    \mathfrak{S}=\sum_{q=1}^{\infty}S(q).
\end{equation*}
Furthermore, on writing
\begin{equation*}
u(\boldsymbol{\gamma})=\int_{-1}^1\int_{-1}^1e(\xi^3\gamma_1+\xi^{2}\eta\gamma_2+ \xi\eta^{2}\gamma_3+\eta^3\gamma_{4}) \,d\xi \, d\eta,
\end{equation*}
we define the singular integral $\mathfrak{J}$ by
\begin{equation*}
\mathfrak{J}=\int_{\R^{4}}\prod_{i=1}^su(c_i\boldsymbol{\gamma}) \, d\boldsymbol{\gamma}.
\end{equation*}
 Then by \cite[Lemma 2.1]{ky2026} one has
 \begin{equation}\label{major}
\int_{\mathfrak{N}_{\delta}} \prod_{i=1}^sF_{c_i}(\boldsymbol{\alpha})\, d\boldsymbol{\alpha} = \mathfrak{S}\mathfrak{J}X^{2s-12}+O(X^{2s-12-\tau}),
    \end{equation}
    for some $\tau>0,$ whenever $s \ge 16$. Furthermore, it follows from \cite[Lemma 5.1]{MR1778504}
that $\mathfrak{S}\mathfrak{J} > 0$ in this instance. 

It remains to bound the contribution from the minor arcs. In the proof of the following lemma, we set up a pruning structure and describe the interpolation between restricted 8th and 20th moments of $F(\balp)$, to which we alluded in the previous section.

\begin{lemma}\label{prop3.2}
    One has
    \begin{equation*}\label{claim in lemma 3.2}
\int_{\mathfrak{n}_{\delta}}|F(\boldsymbol{\alpha})|^{18}d\boldsymbol{\alpha}\ll X^{24-\nu},
    \end{equation*}
    for some $\nu>0,$ where $d\boldsymbol{\alpha}=d\alpha_1 \, d\alpha_2 \, d\alpha_3 \, d\alpha_4.$
\end{lemma}
\begin{proof}
 Recall the definition of $\mathfrak{M}(H)$. We observe that for any $H\leq X^{\delta/100},$ one has
\begin{equation}\label{inclusion}
    \mathfrak{M}(H)^4\subseteq \mathfrak{N}_{\delta}.
\end{equation}
Observe that for $H>0,$ one has
\begin{equation}\label{pruning argument crucial}
\begin{aligned}
   \mathfrak{M}(H)^4 \setminus \mathfrak{M}(H/2)^4
   =\mathfrak{P}_1(H)\cup \mathfrak{P}_2(H)\cup\mathfrak{P}_3(H)\cup \mathfrak{P}_4(H),
\end{aligned}
\end{equation}
where 
\begin{equation*}
    \begin{aligned}
        \mathfrak{P}_1(H)&=\mathfrak{M}(H)^4\setminus(\mathfrak{M}(H/2)\times\mathfrak{M}(H)^3)\\
 \mathfrak{P}_2(H)&=  (\mathfrak{M}(H/2)\times\mathfrak{M}(H)^3)\setminus (\mathfrak{M}(H/2)^2\times\mathfrak{M}(H)^2)\\
  \mathfrak{P}_3(H)&= (\mathfrak{M}(H/2)^2\times\mathfrak{M}(H)^2)\setminus (\mathfrak{M}(H/2)^3\times\mathfrak{M}(H))\\
  \mathfrak{P}_4(H)&=  (\mathfrak{M}(H/2)^3\times\mathfrak{M}(H))\setminus \mathfrak{M}(H/2)^4.\end{aligned}
\end{equation*}
It suffices to show that for $1\leq l\leq 4$ and for all $X^{\delta/100}\leq H\leq X^{3/2}$ with $\delta=10^{-10}$, one has
\begin{equation}\label{claim}
\int_{ \mathfrak{P}_l(H)}|F(\boldsymbol{\alpha})|^{18} \, d\boldsymbol{\alpha}\ll X^{24-\gamma}\ \text{for some}\ \gamma>0.
\end{equation}
In fact, note that by (\ref{inclusion}) one has
\begin{equation}\label{fdsa}
\begin{aligned}
    \mathfrak{n}_{\delta}&\subseteq \mathfrak{M}(X^{3/2})^4\setminus \mathfrak{M}(X^{\delta/100})^4\\
    &\subseteq\bigcup_{j=0}^{\rm L}(\mathfrak{M}(2^{-j}X^{3/2})^4\setminus \mathfrak{M}(2^{-j-1}X^{3/2})^4), 
\end{aligned}
\end{equation}
for some ${\rm L}=O(\log X).$
Then, it follows by $(\ref{pruning argument crucial})$ that 
\begin{equation}\label{asdf1}
\begin{aligned}
    \mathfrak{n}_{\delta}&\subseteq \bigcup_{j=0}^{\rm L}\bigl(\mathfrak{P}_1(2^{-j}X^{3/2})\cup \mathfrak{P}_2(2^{-j}X^{3/2})\cup\mathfrak{P}_3(2^{-j}X^{3/2})\cup \mathfrak{P}_4(2^{-j}X^{3/2})\bigl).
\end{aligned}
\end{equation}
Then, assuming the truth of $(\ref{claim})$, we infer from $(\ref{asdf1})$  that 
\begin{equation*}\label{reasoning}
\begin{aligned}
\int_{\mathfrak{n}_{\delta}}|F(\boldsymbol{\alpha})|^{18}\, d\boldsymbol{\alpha}&\ll \sum_{j=0}^{\rm L}\sum_{l=1}^4\int_{ \mathfrak{P}_l(2^{-j}X^{3/2})}|F(\boldsymbol{\alpha})|^{18} \, d\boldsymbol{\alpha} \ll X^{24-\nu},
\end{aligned}
\end{equation*}
for some $\nu > 0$, as required.
Hence, we turn to verifying $(\ref{claim}).$
Recall the definition (\ref{def of Ll(L;H)}) of $ \mathcal{L}_{l}(L;H)$ with $1\leq l\leq 4.$ We denote $[0,1)^4\setminus \mathcal{L}_{l}(L;H)$ by $(\mathcal{L}_{l}(L;H))^c$, and we set $L=H^{1/7}$. 
Then, we see that
\begin{equation}\label{new dissection}
\begin{aligned}
&\int_{\mathfrak{P}_l(H)}|F(\boldsymbol{\alpha})|^{18} \, d\balp
\\
&=\int_{\mathfrak{P}_l(H)\cap \mathcal{L}_{l}(L;H)}|F(\boldsymbol{\alpha})|^{18}\, d\balp + \int_{\mathfrak{P}_l(H)\cap (\mathcal{L}_{l}(L;H))^c}|F(\boldsymbol{\alpha})|^{18}\, d\balp.
\end{aligned}
\end{equation}
We first consider the case $l=1.$ 
Let us estimate the first term on the right hand side of $(\ref{new dissection})$. By H\"older's inequality, one has
\begin{equation}\label{mean value over the secondary minor arcs}
\begin{aligned} 
I_1 := &\int_{\mathfrak{P}_1(H)\cap \mathcal{L}_{1}(L;H)}|F(\boldsymbol{\alpha})|^{18}\, d\balp \\
    & \le  \left(\int_{\mathfrak{P}_1(H)\cap \mathcal{L}_{1}(L;H)}|F(\boldsymbol{\alpha})|^{20}\, d\balp\right)^{5/6}\left(\int_{\mathfrak{P}_1(H)\cap \mathcal{L}_{1}(L;H)}|F(\boldsymbol{\alpha})|^{8}\, d\balp\right)^{1/6}.
\end{aligned}
\end{equation}
By the definitions of $\mathcal{M}_1(Q)$ and $\mathcal{L}^*_{2,3}(L;H)$, it follows that
\begin{equation}\label{mean value over the secondary minor arcs2}
    \begin{aligned}
     & I_1 \le \left(\int_{\mathcal{M}_1(H/2)\cap  \mathcal{L}_{1}(L;H)}|F(\boldsymbol{\alpha})|^{20}\, d\balp\right)^{5/6}\left(\int_0^1\int_{\mathfrak{M}(H)^2}\int_0^1|F(\boldsymbol{\alpha})|^{8}\, d\balp\right)^{1/6}.
    \end{aligned}
\end{equation}
Now by applying Lemma \ref{improved minor} with $l=1$ and $Q=H/2$, together with (\ref{J10}) and Lemma \ref{k-1 dim measure}, we deduce that
\begin{equation}\label{final bound 1}
    \begin{aligned}
         I_1  &\ll \bigl(X^{8+\varepsilon}J_{10}(2X)\cdot \left(H^{-1}+X^{-1}\right)\cdot L^{-1}\bigr)\bigr)^{5/6}\cdot (X^{10+\varepsilon}\cdot \text{mes}(\mathfrak{M}(H)^2)^{1/6}\\
   & \ll X^{24+\varepsilon}\left(\left(H^{-1}+X^{-1}\right)\cdot L^{-1}\right)^{5/6}\cdot H^{2/3},
    \end{aligned}
\end{equation}
where we have used the fact that $L^{-1}\geq X^{-1/2}$ since $L=H^{1/7}$ and $H\leq X^{3/2}.$

Let us turn to estimating the second term on the right hand side of (\ref{new dissection}). By H\"older's inequality, we have
\begin{equation}\label{mean value over small measure}
\begin{aligned}
    I_2 :=&\int_{\mathfrak{P}_1(H)\cap (\mathcal{L}_{1}(L;H))^c}|F(\boldsymbol{\alpha})|^{18}\, d\balp \\
    & \le  \left(\int_{\mathfrak{P}_1(H)\cap (\mathcal{L}_{1}(L;H))^c}|F(\boldsymbol{\alpha})|^{20}\, d\balp\right)^{5/6}\left(\int_{\mathfrak{P}_1(H)\cap (\mathcal{L}_{1}(L;H))^c}|F(\boldsymbol{\alpha})|^{8}\, d\balp\right)^{1/6}.
\end{aligned}
\end{equation}
Note that 
\begin{equation}\label{P1capL23(L) inclusion}
\begin{aligned}
    \mathfrak{P}_1(H)\cap (\mathcal{L}_{1}(L;H))^c\subseteq  \mathfrak{P}_1(H)\subseteq \mathfrak{m}(H/2)\times [0,1)^3,
\end{aligned}
\end{equation}
by the definition of $\mathfrak{P}_1(H).$ Furthermore,  since \begin{equation*}
    \mathfrak{P}_1(H)\subseteq [0,1)\times \mathfrak{M}(H)\times [0,1)^2
\end{equation*}
and
\begin{equation*}
\begin{aligned}
  &\left([0,1)\times \mathfrak{M}(H)\times [0,1)^2\right)\cap (\mathcal{L}_{1}(L;H))^c \\&=\{\boldsymbol{\alpha}\in [0,1)^4:\ \alpha_2\in \mathfrak{M}(H),\ \alpha_3\in \mathfrak{L}^{(\alpha_2)}(L;H)\},
\end{aligned}
\end{equation*}
we note that
\begin{equation}\label{P1capL23(L)c inclusion}
\begin{aligned}
   &\mathfrak{P}_1(H)\cap (\mathcal{L}_{1}(L;H))^c\\
   &\subseteq \left([0,1)\times \mathfrak{M}(H)\times [0,1)^2\right)\cap (\mathcal{L}_{1}(L;H))^c\\
   &\subseteq [0,1)\times \mathfrak{L}(L;H)\times [0,1),
\end{aligned}
\end{equation}
where $ \mathfrak{L}^{}(L;H)$ is defined in Proposition $\ref{pro5.2}$. Therefore,  it follows by $(\ref{mean value over small measure})$, (\ref{P1capL23(L) inclusion}), and $(\ref{P1capL23(L)c inclusion})$, and  that
\begin{equation}\label{final bound 2}
\begin{aligned}
   I_2  &\leq \left(\int_{\mathfrak{m}(H/2)\times [0,1)^3}|F(\boldsymbol{\alpha})|^{20}\, d\balp \right)^{5/6} \left(\int_0^1\int_{\mathfrak{L}^{}(L;H)}\int_0^1|F(\boldsymbol{\alpha})|^{8}\, d\balp \right)^{1/6}\\
    &\ll ((X^{8+\varepsilon}J_{10}(2X)\cdot\left(H^{-1}+X^{-1}\right))^{5/6}\cdot (X^{10+\varepsilon}\cdot \text{mes}(\mathfrak{L}^{}(L;H)))^{1/6}\\
   & \ll X^{24+\varepsilon}\left(H^{-1}+X^{-1}\right)^{5/6}\cdot (H^{3}L^2)^{1/6},
\end{aligned}
\end{equation}
where we have used Lemma $\ref{basic minor}$, Lemma  \ref{k-1 dim measure}, and Proposition $\ref{pro5.2}.$

Therefore, we conclude from $(\ref{new dissection})$ together with (\ref{final bound 1}) and $(\ref{final bound 2})$ that 
\begin{equation}\label{final bound 6}
\begin{aligned}
&\int_{\mathfrak{P}_1(H)}|F(\boldsymbol{\alpha})|^{18}\, d\balp \ll X^{24+\varepsilon}(\Xi_1+\Xi_2),
\end{aligned}
\end{equation}
where 
$$\Xi_1=\left(\left(H^{-1}+X^{-1}\right)\cdot L^{-1}\right)^{5/6}\cdot H^{2/3}$$
and
$$\Xi_2=\left(H^{-1}+X^{-1}\right)^{5/6}\cdot (H^{3}L^2)^{1/6}.$$
On recalling that $L=H^{1/7}$ and $X^{\delta/100}\leq H\leq X^{3/2},$ one infers that 
\begin{equation}\label{Xibd}
\Xi_i \ll H^{-2/7} + X^{-5/6}H^{23/42} \ll H^{-2/7} + X^{-1/84} \qquad (i=1,2).
\end{equation}
It follows that
\begin{equation}\label{final bound 8}
\int_{\mathfrak{P}_1(H)}|F(\boldsymbol{\alpha})|^{18}\, d\balp \ll X^{24-\gamma},
\end{equation}
for some $\gamma>0$, which completes the case $l=1.$

Next, we consider the case $l=2$. Since the steps are almost the same as in the case $l=1,$ we proceed quickly.  
By following the argument leading from (\ref{mean value over the secondary minor arcs}) to $(\ref{mean value over the secondary minor arcs2})$, we infer that
\begin{equation*}
\begin{aligned}
 I_3 :=&  \int_{\mathfrak{P}_2(H)\cap \mathcal{L}_{2}(L;H)}|F(\boldsymbol{\alpha})|^{18}\, d\balp \\ &\le \left(\int_{\mathcal{M}_2(H/2)\cap \mathcal{L}_{2}(L;H)}|F(\boldsymbol{\alpha})|^{20}\, d\balp \right)^{5/6} \left(\int_0^1 \int_{\mathfrak{M}(H)^2} \int_0^1 |F(\boldsymbol{\alpha})|^{8}\, d\balp\right)^{1/6}.
\end{aligned}
\end{equation*}
By Lemma \ref{improved minor} with $l=2$ and $Q=H/2$, along with (\ref{J10}) and Lemma \ref{k-1 dim measure}, we find just as in (\ref{final bound 1}) that
\begin{equation}\label{final bound 3}
    \begin{aligned}
   I_3 & \ll X^{24+\varepsilon}\left(\left(H^{-1}+X^{-1}\right)\cdot L^{-1}\right)^{5/6}\cdot H^{2/3}.
    \end{aligned}
\end{equation}

Likewise, by following the argument leading from $(\ref{mean value over small measure})$ to $(\ref{final bound 2})$, one infers that 
\begin{equation*}\label{final bound 4}
    \begin{aligned}
  I_4 :=  &\int_{\mathfrak{P}_2(H)\cap (\mathcal{L}_{2}(L;H))^c}|F(\boldsymbol{\alpha})|^{18}\, d\balp \\
     &\leq \left(\int_{[0,1)\times\mathfrak{m}(H/2)\times [0,1)^2}|F(\boldsymbol{\alpha})|^{20}\, d\balp \right)^{5/6} \left(\int_0^1\int_{\mathfrak{L}(L;H)}\int_0^1|F(\boldsymbol{\alpha})|^{8}\, d\balp \right)^{1/6}.
    \end{aligned}
\end{equation*}
Therefore, by applying Lemma $\ref{basic minor}$, Lemma  \ref{k-1 dim measure}, and Proposition $\ref{pro5.2}$ just as in (\ref{final bound 2}), we obtain
\begin{equation}\label{final bound 5}
    \begin{aligned}
  I_4 &\ll X^{24+\varepsilon}\left(H^{-1}+X^{-1}\right)^{5/6}\cdot (H^{3}L^2)^{1/6}.
    \end{aligned}
\end{equation}

Thus, by substituting (\ref{final bound 3}) and $(\ref{final bound 5})$ into the right hand side of $(\ref{new dissection})$, one has
\begin{equation*}
\begin{aligned}
&\int_{\mathfrak{P}_2(H)}|F(\boldsymbol{\alpha})|^{18}\, d\balp \ll X^{24+\varepsilon}(\Xi_1+\Xi_2),
\end{aligned}
\end{equation*}
with $\Xi_1$ and $\Xi_2$ as defined in $(\ref{final bound 6}).$
On recalling (\ref{Xibd}), one sees that
\begin{equation}\label{final bound 7}
\int_{\mathfrak{P}_2(H)}|F(\boldsymbol{\alpha})|^{18}\, d\boldsymbol{\alpha} \ll X^{24-\gamma},
\end{equation}
for some $\gamma>0$, which completes the case $l=2.$ 

It is easily seen that the verifcation $(\ref{claim})$ for the cases $l=3$ and $l=4$ follows a nearly identical pattern. Hence the proof of the lemma is completed by reference to $(\ref{final bound 8})$ and $(\ref{final bound 7})$.
\end{proof}

With Lemma \ref{prop3.2} in hand, it now follows exactly as in the proof of \cite[Lemma 2.2]{ky2026} that whenever $s \ge 18$ one has
\begin{equation}\label{minor}
    \int_{\mathfrak{n}_{\delta}} \prod_{i=1}^s F_{c_i}(\boldsymbol{\alpha}) \, d\boldsymbol{\alpha}\ll X^{2s-12-\eta}
\end{equation}
for some $\eta>0.$ Therefore, by substituting (\ref{major}) and (\ref{minor}) into $(\ref{major and minor arcs dissection})$, we complete the proof of Theorem \ref{cubic theorem}.


\section*{acknowledgement}
The authors would like to express sincere gratitude to Trevor Wooley for suggesting this problem. The second author also gratefully acknowledges support from the KAP allocation at the University of California, Davis. The first author thanks the University of California, Davis for supporting a visit that facilitated the completion  of this work. Finally, the authors thank the referees for useful comments.

\bibliographystyle{alpha}
\bibliography{reference}

\end{document}